\documentclass{amsart}
\usepackage[utf8]{inputenc}
\usepackage{graphicx,amsmath,epsfig,amssymb,amsfonts,amsthm,enumerate,verbatim,todonotes,hyperref,mathtools, dsfont}
\usepackage{mathscinet}
\newtheorem{theorem}{Theorem}[section]
\newtheorem{corollary}[theorem]{Corollary}
\newtheorem{lemma}[theorem]{Lemma}

\newtheorem{conjecture}[theorem]{Conjecture}
\newtheorem{question}[theorem]{Question}

\newtheorem{fact}{Fact}[section]

\theoremstyle{definition}
\newtheorem{definition}[theorem]{Definition}
\newtheorem{claim}[theorem]{Claim}

\theoremstyle{remark}
\newtheorem{remark}[theorem]{Remark}

\newcommand{\restrict}{\mathord{\upharpoonright}}
\newcommand{\restr}[2]{#1\upharpoonright {#2}}
\newcommand{\orpair}[2]{\langle #1,#2 \rangle}

\newenvironment{proofsketch}[1][Proof sketch]{\textbf{#1.} }{\hfill \rule{0.5em}{0.5em}}
\newenvironment{proofclaim}[1][Proof of the claim]{\textbf{#1.} }{\hfill \rule{0.5em}{0.5em}}
\title{$\mathsf{SOCA}$ and $\mathsf{OGA}$ for $\mathsf{HL}$ spaces with strong properties}

\author[Corona-Garc\'{\i}a]{J. A. Corona-Garc\'{\i}a}
\address{Posgrado Conjunto en Ciencias Matem\'aticas UNAM-UMSNH\\Morelia, Michoac\'an\\ M\'exico 58089}
\curraddr{}
\email{jcorona@matmor.unam.mx}
\thanks{{The first author's research has been supported by CONACYT Scholarship 298332, and by the PAPIIT grant IN104419.
}}
\author[Ongay-Valverde]{I. Ongay-Valverde}
\address{Department of Mathematical and Computational Sciences, University of Toronto-Mississaugua, Mississaugua, Ontario, Canada, L5L 1C6}
\curraddr{}
\email{ivan.ongay.valverde@gmail.com}
\thanks{The second author was supported by NSERC grant A-7354}
\author[Ramos-Garc\'{\i}a]{U. A. Ramos-Garc\'{\i}a}
\address{Centro de Ciencias Matem\'aticas\\ Universidad Nacional Aut\'onoma de M\'exico\\ Campus Morelia\\Morelia, Michoac\'an\\ M\'exico 58089}
\curraddr{}
\email{ariet@matmor.unam.mx}
\thanks{The third author was partially supported by the PAPIIT grants IN104419 and IN108122.}


\date{\today}


\begin{document}
\maketitle
\begin{abstract}
We study open colorings in certain classes of hereditary Lindel\"{o}f ($\mathsf{HL}$) spaces and submetrizable spaces. In particular, we show that the definible version for the Open Graph Axiom ($\mathsf{OGA}$) holds for the class of $\mathsf{HL}$ strong Choquet submetrizable spaces extending a well-known result of Feng \cite{MR1113695}. Furthermore, we show the consistency of the Open Graph Axiom for regular spaces that have countable spread and it's square also has it, reaching closer to a well known conjecture of Todor\v{c}evi\'{c} \cite{MR980949}: ``It is consistent that all regular spaces with countable spread satisfy $\mathsf{OGA}$''.
\end{abstract}
\section{Introduction}
Given $X$ a topological space, the Open Graph Axiom of $X$ ($\mathsf{OGA}(X)$) is the assertion that for every partition $[X]^{2}=K_{0}\cup K_{1}$ of (unordered) pairs from $X$ with $K_{0}$ open in the product topology, either there is an uncountable $K_{0}$-homogeneous subset of $X$ or $X$ is the union of countably many $K_{1}$-homogeneous sets \cite{MR1486583}. The Open Graph Axiom ($\mathsf{OGA}$, previously known as $\mathsf{OCA}$ or $\mathsf{OCA}_{[T]}$) is the assertion that if $X$ is a separable metric space, then $\mathsf{OGA}(X)$ holds. $\mathsf{OGA}$ is a consequence of the Proper Forcing Axiom ($\mathsf{PFA}$) \cite{MR980949} and since its introduction in \cite{MR980949} by Todor\v{c}evi\'{c}, it has found a wealth of applications to problems closely related to separable metric spaces (see \cite{MR1752102}, \cite{MR1711328}, \cite{MR2776359}, \cite{MR1900882}, \cite{MR980949}, \cite{MR1999941}, \cite{MR1233818}, and \cite{MR1202874}).

\medskip

One of the open problems related to generalizations of $\mathsf{OGA}$, stated as a conjecture in \cite{MR980949} (repeated in \cite{MR1233818}), asks to weaken the topological assumptions on the space $X$ to, essentially, the best possible:

\begin{conjecture}[Todor\v{c}evi\'{c} \cite{MR980949}]\label{mainConjecture}
It is relatively consistent with $\mathsf{ZFC}$ that if $X$ is a regular space with no uncountable discrete subspace, then 
$\mathsf{OGA}(X)$ holds.
\end{conjecture}

A regular space with no uncountable discrete subspaces (also know as spaces with countable spread) either is hereditary Lindel\"{o}f ($\mathsf{HL}$) or else is hereditary separable ($\mathsf{HS}$). Under $\mathsf{PFA}$ there are no $\mathsf{S}$-spaces \cite{MR980949}. Hence, with $\mathsf{PFA}$, every regular topological space either contains an uncountable discrete subspace or else is $\mathsf{HL}$. Thus, the above conjecture can be reformulated  as follows:

\begin{conjecture}[Todor\v{c}evi\'{c} \cite{MR980949}]\label{mainConjecture-HL}
It is relatively consistent with $\mathsf{ZFC}$ that if $X$ is a regular $\mathsf{HL}$ space, then $\mathsf{OGA}(X)$ holds.
\end{conjecture}

On the other hand, a variant of $\mathsf{OGA}(X)$ which is more relevant to the study of definable sets of reals is the following statement:

\medskip

$\mathsf{OGA}^*(X)$: For every partition $[X]^{2}=K_{0} \cup K_{1}$ such that $K_{0}$ is an open subset of $[X]^{2}$, either 
    \begin{enumerate}
        \item there is a perfect (\emph{i.e.,} a nonempty closed subset without isolated points) $K_{0}$-homogeneous $P \subseteq X$, or 
        \item $X$ is the union of countably many $K_{1}$-homogeneous sets.
    \end{enumerate}

\medskip

It is known that $\mathsf{OGA}^*(A)$ holds for every analytic subset $A$ of any Polish space \cite{MR1113695}. In light of the classical theorem of Choquet which states that a topological space is Polish if and only if it is second countable, $T_{1}$, regular and strong Choquet, a first step that extends the class of separable metrizable spaces (resp., the class of Polish spaces) is the class of all $\mathsf{HL}$ submetrizable spaces (resp., the class of all $\mathsf{HL}$ strong Choquet submetrizable spaces). Recall that a topological space is called \emph{submetrizable} if it has a weaker metrizable topology and is called \emph{strong Choquet} if Player II has a winning strategy in the strong Choquet game \cite{MR0250011,MR1321597}. Using Choquet theorem as a lighthouse, the best possible class that we can work with is the class of $\mathsf{HL}$, $T_{1}$, regular and strong Choquet spaces. For this class, we show that for any space in it and an open partition, either the space is $\sigma$-$K_{1}$-homogeneous or it has a $K_{0}$-homogeneous of size continuum. This property lies between $\mathsf{OGA}^*(X)$ and  $\mathsf{OGA}(X)$. Furthermore, everything points to say that it is the strongest possible result in $\mathsf{ZFC}$ (see Corollary \ref{Corollary:OGA*}).

\medskip

In what follows, we will use the notation $\mathcal{S}$ (resp., $\mathcal{S}^{*}$) to denote the class of all submetrizable spaces (resp., the class of all $\mathsf{HL}$ strong Choquet submetrizable spaces). Examples of non-Polish spaces that belong to $\mathcal{S}^{*}$ are the Sorgenfrey line $\mathbb{S}$ and the space $\omega^{\omega}$ with the Gandy-Harrington topology $\tau_{GH}$ (the topology generated by all lightface $\Sigma^{1}_{1}$ sets) \cite{MR658330,MR2455198}, among others. In this note we consider the following variations on Conjecture \ref{mainConjecture-HL}:

\begin{question}\label{Q:oca}
Is it relatively consistent with $\mathsf{ZFC}$ that $\mathsf{OGA}(X)$ holds for every $X \in \mathcal{S}$?
\end{question}

\begin{question}\label{Q:oca strong epsilon}
Is it relatively consistent with $\mathsf{ZFC}$ that $\mathsf{OGA}(X)$ holds for every regular space $X$ such that both $X$ and $X^2$ have countable spread?
\end{question}

And the corresponding definable version for the first one becomes in:

\begin{question}\label{Q:oca*}
Does $\mathsf{OGA}^*(X)$ holds for every $X\in \mathcal{S}^{*}$?
\end{question}

The study of the statement $\mathsf{OGA}(X)$ outside of the class of separable metrizable spaces was suggested by Todor\v{c}evi\'{c} and Farah in \cite{MR1486583}, but as far as we know, no actual work in the area has been done.

\medskip 

We concentrate on the study of the statement $\mathsf{OGA}(X)$ (resp., $\mathsf{OGA}^{*}(X)$) for every $X \in \mathcal{S}$ (resp., $X\in \mathcal{S}^{*}$) and also in a weakening of $\mathsf{OGA}$ named $\mathsf{SOCA}$ for the consistency results (see Section \ref{OGA and SOCA epsilon strong Consistency}). We answer the Questions \ref{Q:oca strong epsilon} and \ref{Q:oca*} in the positive and, among other results, we provided some partial answers to the Question \ref{Q:oca}.
\medskip

\section{Notation, terminology, and basic definitions}

The notation and terminology in this paper is fairly standard. We will use \cite{MR1321597} as general reference for descriptive set theory, \cite{MR597342} as a reference for set theory and \cite{MR1039321} as a reference for general topology. 

\medskip

The \emph{Souslin operation} respect to a family of subsets  $\mathcal{P}=\{P_s:s\in\omega^{<\omega}\}$ of a set $X$ is the set $\bigcup_{x\in\omega^\omega}\bigcap_{n\in\omega }P_{x\restriction n}$, which is denoted by $\mathcal{A}\mathcal{P}$. If in addition $X$ is a topological space and every $P_{s}$ is closed, then $\mathcal{A}\mathcal{P}$ is called a \emph{Souslin set}. On a different topic, given a topological space $X$, the \emph{strong Choquet game $G_{X}^{s}$} is a two-player game defined as follows: Player I plays a pair $\orpair{x_{n}}{U_{n}}$ where $x_n\in U_n$ and $U_n\subseteq X$ is open. Player II must then play an open set $V_n\subseteq U_n$, with $x_n\in V_n$. This game generates a sequence $U_0\supseteq V_0\supseteq U_1\supseteq V_1\dots$, with $x_n\in V_n$. Player II win this run of the game if $\bigcap_{n\in\omega}U_n=\bigcap_{n\in\omega}V_n\neq\emptyset$. A nonempty space $X$ is called a \emph{strong Choquet space} if Player II has a winning strategy in $G_X^s$.

\medskip

Given a topological space $X$, by $[X]^2$ we denoted the collection of subsets of $X$ of size two. We say that the set $K\subseteq [X]^2$ is \emph{open} if  $\{\orpair{x}{y} \in X^2:\{x,y\}\in K\}$ is open in $X^2$. A partition $[X]^2=K_0\cup K_1$ such that $K_0$ is open is called an  \emph{open partition}. A subset $Y \subseteq X$ is called a  \emph{$K_{i}$-homogeneous set} for $i\in 2$, if $[Y]^{2} \subseteq K_{i}$. We say that $Y$ is a \emph{$\sigma$-$K_{i}$-homogeneous set} for $i\in 2$,  if $Y$ can be covered by countably many $K_{i}$-homogeneous subsets. For $A,B\subseteq X$,  $A\otimes B$ denotes the set $\{\{x,y\}\in [X]^2 \colon \orpair{x}{y} \in A\times B \, \vee \, \orpair{y}{x}\in A\times B \}$. Given a class $\mathcal{X}$ of topological spaces, $\mathsf{OGA}(\mathcal{X})$ (resp., $\mathsf{OGA}^{*}(\mathcal{X})$) abbreviates $\mathsf{OGA}(X)$ for all $X \in \mathcal{X}$ (resp., $\mathsf{OGA}^{*}(X)$ for all $X \in \mathcal{X}$). We will also work with a weakening of  $\mathsf{OGA}$. The Semi-Open Coloring Axiom \emph{$\mathsf{SOCA}(X)$} is the statement ``for every open partition (coloring) $[X]^{2}=K_0\cup K_1$ (where $K_0$ is open) there is an uncountable homogeneous set for either $K_0$ or $K_1$\textquotedblright. \emph{$\mathsf{SOCA}$} abbreviates $\mathsf{SOCA}(X)$ for all metric separable spaces.

\medskip

For Section \ref{OGA and SOCA epsilon strong Consistency}, we will work with various definitions for topological spaces. A collection of sets $\mathcal{A}$ is an \emph{$\omega$-cover of $X$} if and only if for all finite set $F\subseteq X$, there is $A\in \mathcal{A}$ such that $F\subseteq A$. We say that $X$ is an \emph{$\epsilon$-space} if and only if $X^{n}$ is a Lindel\"of space for all $n\in \omega$. This is equivalent to say that every open $\omega$-cover has a countable subcover. This has also been named as finitely powerfully Lindel\"of.

\medskip

Finally, the following definitions will come up in the context of $\mathsf{PFA}$. We say that an space is \emph{Hereditarily Separable  $(\mathsf{HS})$} if all its subspaces have a countable dense set. A space has \emph{countable spread} if and only if all its discrete subspaces are of countable size. As we will mention in Fact \ref{HL or HS}, countable spread has a really close relation with the two following type of spaces.
An \emph{$\mathsf{L}$-space} is a regular topological space that is $\mathsf{HL}$ but not separable and an \emph{$\mathsf{S}$-space} is a regular topological space that is $\mathsf{HS}$ but not Lindel\"of.

\section{Open coloring of pairs from spaces in \texorpdfstring{$\mathcal{S}^{*}$}{Lg}}

In this section we improve Feng's result in \cite{MR1113695} by showing that  $\mathsf{OGA}^{*}(\mathcal{S}^{*})$ is true in $\mathsf{ZFC}$. The proof for this result is, in some sense, similar to the classical proof of $\mathsf{OGA}^{*}(\omega^{\omega})$ given in \cite{MR1486583}.

\medskip

\begin{theorem}\label{T:OGA*}
$\mathsf{OGA}^{*}(\mathcal{S}^{*})$. That is, if $X \in \mathcal{S}^{*}$, then $\mathsf{OGA}^{*}(X)$ holds.
\end{theorem}

\begin{proof}
Fix a metric $\rho$ on $X$ with $\rho \leqslant 1$, where every $\rho$-open ball is open in $X$ and fix an open partition $[X]^{2}=K_{0} \cup K_{1}$ such that $X$ is not a  $\sigma$-$K_{1}$-homogeneous set. Let $\sigma$ be a winining strategy for Player II in the strong Choquet game $G^{s}_{X}$.  Recursively on $s \in 2^{<\omega}$ we will construct three families $\langle x_{s} \colon s \in 2^{<\omega}\rangle$, $\langle U_{s} \colon s \in 2^{<\omega}\rangle$, and $\langle V_{s} \colon \in 2^{<\omega}\rangle$ so that:

\begin{enumerate}
    \item $V_{s} \subseteq U_{s}\subseteq X$ are open sets, $x_{s} \in V_{s}$ and $\mathsf{diam}_{\rho}(U_s) \leqslant 1/2^{|s|}$;
    \item if $U$ is any open set with $x_{s} \in U$, then 
          $U$ is not $\sigma$-$K_{1}$-homogeneous;
    \item the position 
          $\langle \langle x_{s\restriction 0}, U_{s\restriction 0}\rangle, V_{s\restriction 0},\dots$,$\langle x_{s}, U_{s}\rangle, V_{s}\rangle$ is compatible with $\sigma$; and 
    \item $U_{s^\smallfrown0}\otimes U_{s^\smallfrown1}$ is 
          a $K_{0}$-homogenous set with   $\overline{U}_{s^\smallfrown0}^{\rho}\cap \overline{U}_{s^\smallfrown1}^{\rho}=\emptyset$.
\end{enumerate}

\medskip

To guarantee those conditions, the recursion step  needs the following claim:

\begin{claim}\label{Claim:OCA*}
Let $V$ be an open set which is not $\sigma$-$K_{1}$-homogeneous and $\varepsilon > 0$. Then there are $x_0, x_1\in V$ and two open subsets $U_{0},U_{1} \subseteq V$ so that:

\begin{enumerate}[(i)]
    \item $x_i\in U_{i}$ and 
          $\mathsf{diam}_{\rho}(U_{i}) \leqslant \varepsilon$, 
          for $i\in 2$;  
    \item $U_{0}\otimes U_{1}$ is a $K_{0}$-homogeneous set with
          $\overline{U}_{0}^{\rho} \cap \overline{U}_{1}^{\rho}=\emptyset$; and 
    \item if $i \in 2$ and $U$ is any open set with 
          $x_{i} \in U$, then 
          $U$ is not $\sigma$-$K_{1}$-homogeneous.
\end{enumerate} 
\end{claim}

\begin{proofclaim} 
Define $\mathcal{U}$ as the family of all open sets $U$ which are $\sigma$-$K_{1}$-homogeneous. Since $X$ is $\mathsf{HL}$, there exists $\mathcal{U}^{\prime} \subseteq \mathcal{U}$ a countable subfamily with the same union. Then $\bigcup\mathcal{U}^{\prime}$ is $\sigma$-$K_{1}$-homogeneous and $A = V \setminus \bigcup\mathcal{U}^{\prime}$ cannot be covered by countably many  $K_{1}$-homogeneous sets. In particular, $A$ must be an uncountable set and, for every $x \in A$ and every open set $U$ with $x \in U$, it follows that $U$ is not $\sigma$-$K_{1}$-homogeneous.


Fix $\{x_{0},x_{1}\} \in [A]^{2}$ such that $\{x_{0},x_{1}\} \in K_{0}$. Given that $K_0$ is open, there are two disjoint open subsets $U_{0}, U_{1} \subseteq V$ such that $U_{0}\otimes U_{1}\subseteq K_0$ and $\{x_{0},x_{1}\} \in U_{0}\otimes U_{1}$. Finally, if necessary, shrink the sets $U_{0}, U_{1}$ so that the requirements $\overline{U}_{0}^{\rho} \cap \overline{U}_{1}^{\rho}=\emptyset$ and $\mathsf{diam}_{\rho}(U_{i}) \leqslant \varepsilon$ are satisfied for $i\in 2$ .
\end{proofclaim}

\medskip

Let $U_{\emptyset}=X$ and let $x_{\emptyset}$ be any point in 
$U_{\emptyset} \setminus \bigcup\mathcal{U}^{\prime}$, where $\mathcal{U}^{\prime}$ is as in the proof of Claim \ref{Claim:OCA*}. Then, for any open set $U$ with $x_{\emptyset} \in U$ it follows that $U$ is not $\sigma$-$K_{1}$-homogeneous. So $V_{\emptyset}$ is obtained by applying $\sigma$ to the play of the game $G^{s}_{X}$ in which Player I's first move is $\langle x_{\emptyset},U_{\emptyset}\rangle$.


In the recursive step assume $x_{s},$  $U_{s}$, and $V_{s}$ have been defined so as to satisfy $(1)$ to $(3)$. Applying the previous claim to $V_{s}$ and $\varepsilon = 1/2^{|s|+1}$ we get two points $x_{s^\smallfrown 0}, x_{s^\smallfrown 1} \in V_{s}$ and two open subsets $U_{s^\smallfrown 0}, U_{s^\smallfrown 1} \subseteq V_{s}$ satisfying the conclusion of Claim \ref{Claim:OCA*}. Then $V_{s^\smallfrown i}$ for $i \in 2$ is obtained by applying $\sigma$ to the play of the game $G^{s}_{X}$ in which the following moves been made: 
\[
\langle x_{\restr{s}{0}},U_{\restr{s}{0}} \rangle, V_{\restr{s}{0}}, \dots, \langle x_{s}, U_{s}\rangle, V_{s}, \langle x_{s^\smallfrown i},  U_{s^\smallfrown i}\rangle.
\]
This finishes our recursive construction of the three families mentioned above.

\medskip

Since $\sigma$ is winning for II and the sequence 
\[
\langle \langle x_{\restr{s}{0}},U_{\restr{s}{0}} \rangle, V_{\restr{s}{0}}, \dots, \langle x_{s}, U_{s}\rangle, V_{s},\dots \rangle \in [\sigma]
\]
for $a \in 2^{\omega}$, it follows that $\bigcap_{n \in \omega}U_{\restr{a}{n}} = \bigcap_{n \in \omega}V_{\restr{a}{n}} \neq \emptyset$ for each $a \in 2^{\omega}$. Moreover, $\bigcap_{n \in \omega}U_{\restr{a}{n}} = \bigcap_{n \in \omega} \overline{U}_{\restr{a}{n}}^{\rho}$ is a singleton, say $\{f(a)\}$. Clearly, $f \colon 2^{\omega} \to \langle X,\rho\rangle$ is injective and continuous, and therefore an embedding. Then $F=f^{\prime\prime} 2^{\omega}$ is a $K_{0}$-homogeneous closed subset of $X$ of size the continuum. Finally, since $X$ is $\mathsf{HL}$, 
\[
P=\{x \in F \colon x \text{ is a condensation point of } F\}\footnote{A point $x$ is a \emph{condensation point} if every open neighborhood of $x$ is uncountable.}
\]
is the desired perfect $K_{0}$-homogeneous subset of $X$.
\end{proof}
 
It is important to discuss the topological property asked for the $K_{0}$-homogeneous set in Theorem \ref{T:OGA*}. While in the context of Polish spaces, one can improve the $K_{0}$-homogeneous set to be perfect and compact, in our context that is not possible. For example, it is known that all the compact spaces of the Sorgenfrey line $\mathbb{S}$ are countable. Nevertheless, this topological space satisfy all the conditions of the previous theorem.

\medskip

To finish our study of all the strong Choquet spaces of our interest, we need to drop the assumption of submetrizable. In the same way that going from Polish to submetrizable weaken the $K_{0}$-homogeneous set from compact to perfect, when we drop submetrizablility the $K_{0}$-homogeneous set may not be closed.

\begin{corollary}\label{Corollary:OGA*}
Given a $\mathsf{HL}$, $T_{1}$, regular and strong Choquet space $X$. We have $\mathsf{OGA}(X)$. Furthermore, if $X$ is not $\sigma$-$K_1$-homogeneous, it has a continuum size $K_0$-homogeneous set without isolated points.
\end{corollary}

\begin{proof}
In the proof of Theorem \ref{T:OGA*}, we only use the condition of submetrizable to have a metric that ensures that $\bigcap_{n \in \omega}U_{\restr{a}{n}}=\{f(a)\}$ for every $a\in 2^{\omega}$. This makes $f$ a continuous function.


In the absence of such metric, strong Choquet still ensures that $\bigcap_{n \in \omega}U_{\restr{a}{n}}\neq \emptyset$ but the intersection might not be a singleton. Using the Axiom of Choice, we can select $g(a)\in \bigcap_{n \in \omega}U_{\restr{a}{n}}$ for all $a\in 2^{\omega}$. This function $g$ is an injective but might not be continuous. Therefore, $G=g^{\prime\prime}2^{\omega}$ is size continuum $K_{0}$-homogeneous set and we could find, using $\mathsf{HL}$, a continuum size subset of $G$ without isolated points. Nevertheless, we cannot show that $G$ contains an uncountable closed subset.
\end{proof}

The following result shows how the elements of $\mathcal{S}^{*}$ can be studied as submetrizables extentions of the Baire space $\omega^{\omega}$.  

\begin{theorem}\label{Th:Continuous image of Baire}
 Let $\langle X, \tau \rangle$ be any element of $\mathcal{S}^{*}$ with $\rho$ a metric on $X$ witness that $\tau$ is submetrizable. Then 
 \begin{enumerate}
  \item $\tau\subseteq \Sigma_{1}^{1}\left(\langle \hat{X}, \hat{\rho}\rangle\right)$, \emph{i.e.,} every open set in $\tau$ is absolutely analytic in its metric completion.     
  \item There is a topology $\tau_{X}$ on $\omega^{\omega}$ extending the usual topology on $\omega^{\omega}$ so that $\langle \omega^{\omega}, \tau_{X}\rangle \in \mathcal{S}^{*}$ and $\langle X,\tau \rangle$ is a continuous image of $\langle \omega^{\omega}, \tau_{X}\rangle$.
\end{enumerate}
\end{theorem}

\begin{proof}

To prove $(1)$, let $A \in \tau$. We will show that $A$ is the Souslin operation of a family of closed sets in $\langle \hat{X},\hat{\rho}\rangle$. Since $A$ is an open set in the strong Choquet space $\langle X,\tau \rangle$, the subspace $\langle A, \restr{\tau}{A}\rangle$ is also a strong Choquet space. Let $\sigma$ be a winning strategy for Player II in the game $G^{s}_{A}$. Recursively on $t \in \omega^{<\omega}$ we construct a subtree $S \subseteq \sigma$ of sequences of the form $s_{t}:=\langle \langle x_{\restr{t}{1}}, U_{\restr{t}{1}}\rangle, V_{\restr{t}{1}},\dots$,$\langle x_{t}, U_{t}\rangle, V_{t}\rangle$ for $t \in \omega^{<\omega} \setminus \{\emptyset\}$ so that: 

\begin{enumerate}[(i)]
    \item $A =\bigcup_{n \in \omega} V_{\langle n \rangle}$;  
    \item $V_{t} = \bigcup_{n \in \omega} V_{t^{\smallfrown} n}$ 
          and $\mathsf{diam}_{\rho}(U_{t}) < 1/2^{|t|}$, for $t \in \omega^{<\omega} \setminus \{\emptyset\}$.
\end{enumerate}

\medskip

To define $s_{t}=\langle \langle x_{t}, U_{t}\rangle,V_{t}\rangle$ for $|t| = 1$, consider elements $\langle \langle x_{0}, U_{0}\rangle,V_{0}\rangle \in \sigma$ where $\mathsf{diam}_{\rho}(U_{0}) < 2^{-1}$. Since $x_{0} \in A$ can be arbitrarily chosen, the collection of $V_{0}$ components form an open cover of $A$. Then, as $\tau$ is $\mathsf{HL}$, there exists a countable subcover $\mathcal{V}_{0}$ consisting of such $V_{0}$ components. Let $\{V_{\langle n \rangle} \colon n \in \omega\}$ be an enumeration of $\mathcal{V}_{0}$. So for any $n \in \omega$, one can find $\langle x_{\langle n \rangle}, U_{\langle n \rangle}\rangle$ such that  
$\langle \langle x_{\langle n \rangle}, U_{\langle n \rangle}\rangle,V_{\langle n \rangle}\rangle \in \sigma$. Thus, we put 
$\langle \langle x_{\langle n \rangle}, U_{\langle n \rangle}\rangle,V_{\langle n \rangle}\rangle \in S$ for $n \in \omega$. In general, if $s_{t}=\langle \langle x_{\restr{t}{1}}, U_{\restr{t}{1}}\rangle, V_{\restr{t}{1}},\dots$,$\langle x_{t}, U_{t}\rangle, V_{t}\rangle$ has been defined so as to satisfy (i) and (ii), then we consider all elements $\langle \langle x, U\rangle,V\rangle$ so that $s_{t}^{\smallfrown}\langle \langle x, U\rangle,V\rangle \in \sigma$ 
where $\mathsf{diam}_{\rho}(U) < 1/2^{|t|+1}$. As before, since $x \in V_{t}$ can be arbitrarily chosen, the family of $V$ components form an open cover of $V_{t}$. Again, as $\tau$ is $\mathsf{HL}$, there is a countable subcover $\mathcal{V}$ consisting of such $V$ components. Let 
$\{ V_{t^{\smallfrown} n} \colon n \in \omega\}$ be an enumeration of $\mathcal{V}$. Then for any $n \in \omega$, one can find $\langle x_{t^{\smallfrown} n}, U_{t^{\smallfrown} n}\rangle$ such that
$s_{t}^{\smallfrown}\langle \langle x_{t^{\smallfrown} n}, U_{t^{\smallfrown} n}\rangle,V_{t^{\smallfrown} n}\rangle \in \sigma$. Therefore, we put $s_{t}^{\smallfrown}\langle \langle x_{t^{\smallfrown} n}, U_{t^{\smallfrown} n}\rangle,V_{t^{\smallfrown} n}\rangle \in S$ for $n \in \omega$.

\medskip

Let $\{P_{t} \colon  t \in \omega^{<\omega}\}$ be such that $P_{\emptyset} = \overline{U}_{\emptyset}^{\rho}$ where $U_{\emptyset}:=A$, and $P_{t}=\overline{U}_{t}^{\rho}$ for $t \in \omega^{<\omega} \setminus \{\emptyset\}$. Then for any $t \in \omega^{<\omega}$, there is a closed subset $\hat{P}_{t} \subseteq \hat{X}$ such that $P_{t} = \hat{P}_{t} \cap X$. We claim that $A$ is obtained from 
$\hat{\mathcal{P}}=\{\hat{P}_{t} \colon t \in \omega^{<\omega}\}$ by 
the Souslin operation, \emph{i.e.,} $A = \mathcal{A}\hat{\mathcal{P}} = \bigcup_{z \in \omega^{\omega}}\bigcap_{n \in \omega} \hat{P}_{\restr{z}{n}}$. First we show that $A \subseteq \mathcal{A}\hat{\mathcal{P}}$. Let $x \in A$. By our construction, $\{U_{t} \colon t \in \omega^{n}\}$ form an open cover of $A$ for every $n \in \omega$. Then there exists $z \in \omega^{\omega}$ such that 
\[
x \in \bigcap_{n \in \omega} U_{\restr{z}{n}} \subseteq \bigcap_{n \in \omega}  
\overline{U}_{\restr{z}{n}}^{\rho}=\bigcap_{n \in \omega}P_{\restr{z}{n}}  \subseteq \bigcap_{n \in \omega} \hat{P}_{\restr{z}{n}}.
\]
For the converse suppose that $x \in \bigcap_{n \in \omega} \hat{P}_{\restr{z}{n}}$ for some $z \in \omega^{\omega}$. Notice that  
\[
\mathsf{diam}_{\hat{\rho}}(\hat{P}_{\restr{z}{n}}) = \mathsf{diam}_{\rho}(P_{\restr{z}{n}}) = \mathsf{diam}_{\rho}(U_{\restr{z}{n}}) < 1/2^{n}
\]
for $n \geqslant 1$. Then, $\{x\} = \bigcap_{n \in \omega} \hat{P}_{\restr{z}{n}}$. On the other hand, since the construction followed a winning strategy of Player II in the game $G^{s}_{A}$, it follows that $\bigcap_{n \in \omega} U_{\restr{z}{n}} \neq \emptyset$. But $\bigcap_{n \in \omega} U_{\restr{z}{n}} \subseteq \bigcap_{n \in \omega} \hat{P}_{\restr{z}{n}}$, so $x \in \bigcap_{n \in \omega} U_{\restr{z}{n}}$ and thus $x \in A$.     

\medskip

To prove $(2)$, by virtue of $(1)$, $\langle X, \restr{\hat{\rho}}{X^{2}}\rangle=\langle X, \rho \rangle$ is analytic in $\langle \hat{X},\hat{\rho} \rangle$. So there is a continuous surjection $f \colon \langle \omega^{\omega},\tau_{P} \rangle \to \langle X,\rho \rangle$ where $\tau_{P}$ is the usual Polish topology on $\omega^{\omega}$. Then we consider the topology $\tau_{X}$ on $\omega^{\omega}$ generated by $\tau_{P} \cup \{f^{-1}[U]:U \in \tau\}$. We claim that $\langle \omega^{\omega}, \tau_{X} \rangle \in \mathcal{S}^{*}$. Indeed, it is clear that $\tau_{X}$ is a submetrizable topology. Since $\tau_{P}$ is second countable and $\tau$ is $\mathsf{HL}$, it easy to see that $\tau_{X}$ is also $\mathsf{HL}$.

\medskip

We check now that $\tau_{X}$ is strong Choquet. Let $\sigma^{\prime}$ and  $\sigma^{\prime \prime}$ be winning strategies for Player II in the games $G^{s}_{\langle X,\tau \rangle}$ and $G^{s}_{\langle \omega^{\omega},\tau_{P} \rangle}$ respectively. We describe a winning strategy $\sigma$ for Player II in the game $G^{s}_{\langle \omega^{\omega},\tau_{X} \rangle}$: Player I starts with $\langle x_{0},U_{0}\rangle$. Without loss of generality we may assume that 
$U_{0}=N_{s_{0}} \cap f^{-1}[U_{0}^{\prime}]$ for some $U_{0}^{\prime} \in \tau$ with $\mathsf{diam}_{\rho}(U_{0}^{\prime}) < 1$ and $s_{0} \in \omega^{<\omega}$.\footnote{For $s \in \omega^{<\omega}$ we let $N_{s}=\{x \in \omega^{\omega} \colon s \subseteq x\}$.} In particular, $f(x_{0}) \in U_{0}^{\prime}$ and $x_{0} \in N_{s_{0}}$. Let $V_{0}^{\prime}$ (resp., $V_{0}^{\prime\prime}$) be Player II's response according to $\sigma^{\prime}$ (resp., $\sigma^{\prime\prime}$) when Player I's first move is $\langle f(x_{0}), U_{0}^{\prime}\rangle$ (resp., $\langle x_{0},N_{s_{0}}\rangle$) in $G^{s}_{\langle X, \tau \rangle}$ (resp., $G^{s}_{\langle \omega^{\omega},\tau_{P} \rangle}$). Then in the game $G^{s}_{\langle \omega^{\omega},\tau_{X} \rangle}$ we let $V_{0}:= V_{0}^{\prime\prime} \cap f^{-1}[V_{0}^{\prime}]$ be Player II's first response. In general, suppose $\langle x_{0}, U_{0}\rangle, V_{0}, \dots, \langle x_{n}, U_{n}\rangle, V_{n}$, and $U_{n}^{\prime}, V_{n}^{\prime}, V_{n}^{\prime\prime}, s_{n}$ are given so that 

\begin{enumerate}[(i)]
    \item $U_{n}=N_{s_{n}} \cap f^{-1}[U_{n}^{\prime}]$ for some 
          $s_{n} \in \omega^{<\omega}$ and $U_{n}^{\prime} \in \tau$ with 
          $\mathsf{diam}_{\rho}(U_{n}^{\prime}) < 2^{-n}$;
    \item $V_{n}= V_{n}^{\prime\prime} \cap f^{-1}[V_{n}^{\prime}]$; and  \item $\langle \langle x_{0}, U_{0}\rangle, V_{0}, \dots, \langle x_{n}, U_{n}\rangle, V_{n}\rangle \in \sigma$, $\langle \langle f(x_{0}), U_{0}^{\prime}\rangle, V_{0}^{\prime}, \dots, \langle f(x_{n}), U_{n}^{\prime}\rangle, V_{n}^{\prime}\rangle \in \sigma^{\prime}$, and $\langle \langle x_{0}, N_{s_{0}}\rangle, V_{0}^{\prime\prime}, \dots, \langle x_{n}, N_{s_{n}}\rangle, V_{n}^{\prime\prime}\rangle \in \sigma^{\prime\prime}$.    
\end{enumerate}

Also, assume that in the $(n+1)$-th round Player I has played $\langle x_{n+1}, U_{n+1}\rangle$ in the game $G^{s}_{\langle \omega^{\omega},\tau_{X} \rangle}$. Without loss of generality, we may assume that $U_{n+1}= N_{s_{n+1}} \cap f^{-1}[U_{n+1}^{\prime}] \subseteq V_{n}$ where $s_{n+1} \in \omega^{<\omega}$ and $U_{n+1}^{\prime} \in \tau$ with $2^{-|s_{n+1}|}, \mathsf{diam}_{\rho}(U_{n+1}^{\prime})$ $< 2^{-(n+1)}$. Then in the game  
$G^{s}_{\langle \omega^{\omega},\tau_{X} \rangle}$ (resp., $G^{s}_{\langle \omega^{\omega},\tau_{P} \rangle}$) let Player I play the move $\langle f(x_{n+1}),U_{n+1}^{\prime} \rangle$ (resp., $\langle x_{n+1}, N_{s_{n+1}}\rangle$) in the $(n+1)$-th round. Suppose $V_{n+1}^{\prime}$ (resp., $V_{n+1}^{\prime\prime}$) is Player II's response in the $(n+1)$-th round in the game $G^{s}_{\langle \omega^{\omega},\tau_{X} \rangle}$ (resp., $G^{s}_{\langle \omega^{\omega},\tau_{P} \rangle}$) according to her winning strategy. Then in the game $G^{s}_{\langle \omega^{\omega},\tau_{X} \rangle}$ let Player II's response in the $(n+1)$-th round be $V_{n+1:}=V_{n+1}^{\prime\prime} \cap f^{-1}[V_{n+1}^{\prime}]$. It is then easy to see that the rules of the game are followed, \emph{i.e.,} $x_{n+1} \in V_{n+1} \subseteq U_{n+1}$. Moreover, by construction 
\[
\bigcap_{n \in \omega} V_{n} = \left(\bigcap_{n \in \omega}N_{s_{n}}\right) \cap \left(\bigcap_{n \in \omega}f^{-1}[V_{n}^{\prime}]\right).
\]
Now, since the sequence $\langle \langle f(x_{0}), U_{0}^{\prime}\rangle, V_{0}^{\prime}, \dots, \langle f(x_{n}), U_{n}^{\prime}\rangle, V_{n}^{\prime} \dots \rangle$ followed the winning strategy $\sigma^{\prime}$ and $\mathsf{diam}_{\rho}(V_{n}^{\prime}) \to 0$, we get that $\bigcap_{n \in \omega} U_{n}^{\prime}=\bigcap_{n \in \omega} V_{n}^{\prime}$ is a singleton, say $\{y\}$. Then $f^{-1}[\{y\}]=\bigcap_{n \in \omega}f^{-1}[V_{n}^{\prime}]$ is a closed subset of $\langle \omega^{\omega},\tau_{P} \rangle$. On the other hand, since the sequence $\langle \langle x_{0}, N_{s_0}\rangle, V_{0}^{\prime\prime}, \dots, \langle x_{n}, N_{s_n}\rangle, V_{n}^{\prime\prime} \dots \rangle$ followed the winning strategy $\sigma^{\prime\prime}$, we have that $\bigcap_{n\in \omega}N_{s_n}\neq\emptyset$. Furthermore, $N_{s_{n}} \cap f^{-1}[\{y\}]\neq\emptyset$ for all $n \in \omega$ and, since $f^{-1}[\{y\}]$ is closed, we conclude
\[
\bigcap_{n \in \omega} U_{n} = \bigcap_{n \in \omega} V_{n} = \bigcap_{n \in \omega} \left(N_{s_{n}} \cap f^{-1}[\{y\}]\right) \neq \emptyset.
\]
This shows that the strategy we just described is winning for Player II.
\end{proof}

It is well known that $\mathsf{OGA}(X)$ fails when $X$ has an uncountable set of isolated points (\cite{MR1486583}, p. 79). The following result is a version of this phenomena for a class of submetrizable spaces with uncountable discrete subspaces.

\begin{theorem}
Let $X$ be a submetrizable space with $\rho$ a metric on $X$ witness for the submetrizability of $X$ so that $\langle X,\rho \rangle$ is a separable space. Suppose that there is an uncountable discrete subspace $D \subseteq X$ such that $\overline{D}=\overline{D}^{\rho}$. Then $\mathsf{OGA}(X)$ does not hold.
\end{theorem}

\begin{proof}
Without loss of generality we may assume that $D$ has, at most, the size of continuum. For every $x \in D$, let $U_{x}$ be an open set so that $U_{x} \cap D = \{x\}$. Given $\varepsilon > 0$, we put $V(x,\varepsilon) := U_{x} \cap B(x,\varepsilon)$, where $B(x,\varepsilon) = \{y \in X \colon \rho(x,y) < \varepsilon\}$. 

\medskip

Consider the Sierpiński coloring $c \colon [D]^{2} \to 2$ associated with the set $D$, \emph{i.e.,} the coloring $c$ has no uncountable homogeneous set. Let $Z = \{\{x,y\} \in [D]^{2} \colon c(\{x,y\})$ $=0\}$. For each $\{x,y\} \in Z$, define the following open set in $[X]^{2}$:

\begin{align*}
C_{\{x,y\}} &=\left\{\{p,q\} \in [X]^{2} \colon p \in V\left(x,\frac{\rho(x,y)}{4}\right) \, \wedge \, q \in V\left(y,\frac{\rho(x,y)}{4}\right)\right\}\\
&=V\left(x,\frac{\rho(x,y)}{4}\right) \otimes V\left(y,\frac{\rho(x,y)}{4}\right).    
\end{align*}
Applying the triangle inequality, it is easy to see that if $\{p,q\} \in C_{\{x,y\}}$, then 
\[
\rho(p,q) > \frac{\rho(x,y)}{2}.
\]
Define a partition $[X]^{2} = K_{0} \cup K_{1}$ by 
\[
\{p,q\} \in K_{0} \text{ iff there exists } \{x,y\} \in Z \text{ such that } \{p,q\} \in C_{\{x,y\}}. 
\]
It is clear that this is an open partition of $[X]^{2}$ and since $c$ has no uncountable $1$-homogeneous set, it follows that $X$ is not $\sigma$-$K_{1}$-homogeneous.

\medskip

To see that $X$ has no uncountable $K_{0}$-homogeneous set, assume toward a contradiction this is not the case and take an uncountable $K_{0}$-homogeneous set $H$. Since $c$ has no uncountable $0$-homogeneous, we have that $|H \cap D|\leqslant \aleph_{0}$. Thus, taking a subset if necessary, we can assume that $H \cap D = \emptyset$. Moreover, we can assume that $H \cap \overline{D} = \emptyset$. Indeed, suppose that there is a $p \in H \cap \overline{D}$ and choose any point $q \in H \setminus \{p\}$. Then there exists $\{x,y\} \in Z$ such that 
\[
\{p,q\} \in V\left(x,\frac{\rho(x,y)}{4}\right) \otimes V\left(y,\frac{\rho(x,y)}{4}\right) \subseteq K_{0}.
\]
Say $p \in V\left(x,\frac{\rho(x,y)}{4}\right)$. By definition, $V\left(x,\frac{\rho(x,y)}{4}\right) \cap D = \{x\}$. Then since $p \notin D$, we have that $p \in V\left(x,\frac{\rho(x,y)}{4}\right) \setminus \{x\}$ and, as $\left(V\left(x,\frac{\rho(x,y)}{4}\right) \setminus \{x\}\right) \cap D = \emptyset$ we arrive at a contradiction to the assumption that $p \in \overline{D}$.

\medskip

Now, since $\overline{D}=\overline{D}^{\rho}$, it follows that
\[
\overline{D} = \bigcap_{n \geqslant 1} B(D,1/n),
\]
where $B(D,1/n)=\{x \in X \colon \rho(x,D) < 1/n\}$. Thus, by our assumption about the set $H$, we conclude that there is an $n \geqslant 1$ such that $H_{n} := H \cap (X \setminus B(D,1/n))$ is an uncountable set. Now, for $\{p,q\} \in [H_{n}]^{2}$, there exists $\{x,y\} \in Z$ such that $\{p,q\} \in C_{\{x,y\}}$. Then, $\frac{1}{n} \leqslant \rho(x,p),\rho(y,q) < \frac{\rho(x,y)}{4}$. This implies in particular that $\frac{2}{n} < \frac{\rho(x,y)}{2} < \rho(p,q)$. Therefore, $H_{n}$ form an uncountable discrete subspace of $\langle X, \rho \rangle$, which contradicts the fact that $\langle X, \rho \rangle$ is a separable space.
\end{proof}

We have the following immediate corollary.  
 
\begin{corollary}
$\mathsf{OGA}(N)$ and  $\mathsf{OGA}(\mathbb{S} \times \mathbb{S})$ fails
for the Niemytzki plane $N$ and the Sorgenfrey plane $\mathbb{S} \times \mathbb{S}$.\hfill $\square$
\end{corollary}

The end of this section is dedicated to create tools towards a possible answer to Question \ref{Q:oca}. For this, we define the countable c.d.$\pi$-base property. Theorem \ref{ocasor} is inspired by a suggestion of Yinhe Peng for a proof of Corollary \ref{equivalencesorgenfrey}.

\begin{definition}
We say that $X$ has a \emph{countable difference $\pi$-base of countable size} (countable c.d.$\pi$-base) if there is a countable family $\mathcal{V}$ of open sets such that $|U \setminus \mathsf{Int}_{\mathcal{V}}(U)| \leqslant \aleph_{0}$ for every open set $U$, where $\mathsf{Int}_{\mathcal{V}}(U) := \bigcup \{ V \in \mathcal{V} \colon V \subseteq U\}$.
\end{definition}

Notice that if $X$ has a countable c.d.$\pi$-base with dispersion character $\Delta(X) > \aleph_{0}$, then $X$ is necessary a separable space.\footnote{Recall that the \emph{dispersion character} $\Delta(X)$ of a topological space $X$ is the minimum cardinality of a nonempty open set.}

\begin{theorem}[$\mathsf{OGA}$]\label{ocasor}
 Let $X$ be a $\mathsf{HL}$ Hausdorff space and suppose that it has a countable c.d.$\pi$-base. Then $\mathsf{OGA}(X)$ is true.
\end{theorem}

\begin{proof}
Let $[X]^{2}=K_{0} \cup K_{1}$ be an open partition. Fix $\mathcal{V}$ a countable family of open sets witness that $X$ has a countable c.d.$\pi$-base. Let $\{V_{n} \colon n \in \omega\}$ be an enumeration of $\mathcal{V}$. Define $A_{n}:=\{x \in X \colon \{x\}\otimes V_{n}\subseteq K_{0}\}$ and  $K_{0}^{\prime} := \bigcup_{n \in \omega} A_{n} \otimes V_{n}$. Note that $K_{0}^{\prime} \subseteq K_{0}$. Consider the topology $\tau^{\prime}$ on $X$ generated by $\mathcal{V} \cup \{A_{n} \colon n \in \omega\}$. It is clear that $\tau^{\prime}$ is a second countable Hausdorff topology and $K_{0}^{\prime}$ is an open set in $[\langle X,\tau^{\prime} \rangle]^{2}$. Thus, $[X]^{2} = K_{0}^{\prime} \cup K_{1}^{\prime}$ define an open partition in $[\langle X,\tau^{\prime} \rangle]^{2}$ with $K_{1}^{\prime} = [X]^{2} \setminus K_{0}^{\prime}$. By the $\mathsf{OGA}$, we have the following alternatives:

\medskip

$(1)$ There is an uncountable $K_{0}^{\prime}$-homogeneous set $H$.
      Then, $[H]^{2} \subseteq K_{0}^{\prime} \subseteq K_{0}$ and hence $H$ is also an uncountable $K_{0}$-homogeneous set.

\medskip
          
$(2)$ There is a family $\{X_{n} \colon n \in \omega\}$ such that 
      $X=\bigcup_{n \in \omega} X_{n}$ and $[X_{n}]^{2} \subseteq K_{1}^{\prime}$ for all $n \in \omega$. Since $X$ is $\mathsf{HL}$, each $X_{n}$ can be written as $P_{n} \cup C_{n}$, with $P_{n}=\{ x \in X_{n} \colon x \text{ is a condensation point of } X_{n}\}$ and $C_{n}$ is a countable open set in $X_{n}$. Then 
      \[
        X = \left( \bigcup_{n \in \omega} P_{n}\right) \cup 
            \left( \bigcup_{n \in \omega} C_{n}\right).
      \]
      Thus, without loss of generality we may assume that each $x \in X_{n}$ is a condensation point of $X_{n}$. We claim that $X_{n} \subseteq K_{1}$ for every $n \in \omega$. Indeed, assume toward a contradiction that there is an $n \in \omega$ such that $[X_{n}]^{2} \cap K_{0} \neq \emptyset$. Let $\{x, y\} \in [X_{n}]^{2} \cap K_{0}$. So, there are disjoint open sets $U_{x}$ and $U_{y}$ so that $\{x,y\} \in U_{x} \otimes U_{y} \subseteq K_{0}$. Define $V_{x}:=\mathsf{Int}_{\mathcal{V}}(U_{x} \setminus \{x\})$ (resp., $V_{y}:=\mathsf{Int}_{\mathcal{V}}(V_{y} \setminus \{y\})$). Since $y$ is a condensation point of $X_{n}$, $|V_{y} \cap X_{n}| > \aleph_{0}$. We take 
      $z \in V_{y} \cap X_{n}$ and $V_{m} \in \mathcal{V}$ such that $z \neq y$ and $z \in V_{m} \subseteq V_{y}$. Then, $\{x\} \otimes V_{m} \subseteq U_{x} \otimes U_{y} \subseteq K_{0}$ and hence $x \in A_{m}$. By definition, $A_{m} \otimes V_{m} \subseteq K_{0}^{\prime}$ and, as $\{x,z\} \in A_{m} \otimes V_{m}$ we arrive at a contradiction to the assumption that $X_{n}$ is a $K_{1}^{\prime}$-homogeneous set.
\end{proof}

\begin{definition}
Given $X \in \mathcal{S}^{*}$ we denoted by \emph{$\mathsf{OGA}_{X}$} the statement: $\mathsf{OGA}(Y)$ holds for every subspace $Y \subseteq X$. 
\end{definition}

\begin{fact}[\cite{MR1486583} 10.1]\label{Fact:OGA reduction}
Let $X$ and $Y$ be topological spaces so that $Y$ is $T_{0}$ and there is a continuous surjection function $f \colon X \to Y$. Then $\mathsf{OGA}(X)$ implies $\mathsf{OGA}(Y)$.\hfill $\square$
\end{fact}

Thus, $\mathsf{OGA}_{\omega^{\omega}}$ is equivalent to $\mathsf{OGA}$. With this notation we have the following corollary of the previous theorem.

\begin{corollary}\label{equivalencesorgenfrey}
The statements $\mathsf{OGA}$ and $\mathsf{OGA}_\mathbb{S}$ are equivalent.
\end{corollary}
 
\begin{proof}
 Since the identity $\mathsf{id}_{\mathbb{R}} \colon \mathbb{R} \to \mathbb{R}$ is also continuous with respect to the Sorgenfrey line $\mathbb{S}$, it follows that $\mathsf{OGA}_{\mathbb{S}}$ implies  $\mathsf{OGA}_{\mathbb{R}}$. For the other implication, note that every subspace of $\mathbb{S}$ has a countable c.d.$\pi$-base.
\end{proof}



\section{Consistency of \texorpdfstring{$\mathsf{SOCA}$}{Lg} and \texorpdfstring{$\mathsf{OGA}$}{Lg} for finitely powerfully \texorpdfstring{$\mathsf{HL}$}{Lg} spaces}\label{OGA and SOCA epsilon strong Consistency}

Although our focus during the previous section was to work with the definable version of $\mathsf{OGA}$, we also asked ourselves: What is the limit of the usual forcing proofs of $\mathsf{OGA}$ and $\mathsf{SOCA}$?

\medskip

We find an interesting class of spaces for which the proof of the consistency of $\mathsf{SOCA}$ done by Abraham-Rudin-Shelah in \cite{MR801036} and the proposed proof of $\mathsf{OGA}$ by Kunen \cite{MR2905394} also work. We are not sure if this is the actual limit of the technique, but we believe that is a considerable advance in this direction.

\medskip

We start this section with new definitions:

\begin{definition}
We say that $X$ is a \emph{strong $\epsilon$-space} if and only if $X^{n}$ is $\mathsf{HL}$ space for all $n\in \omega$. We use the name \emph{finitely powerfully $\mathsf{HL}$} for the title of section.
\end{definition}

We will use the  expression \emph{$\mathsf{SOCA}(s\text{-}\epsilon)$} and \emph{$\mathsf{OGA}(s\text{-}\epsilon)$} to abbreviate $\mathsf{SOCA}(X)$ and $\mathsf{OGA}(X)$ for all $X$ that are Hausdorff strong $\epsilon$-spaces, respectively.

\medskip

It is a known result, which Tall and Scheepers strengthened in \cite{MR2720214,MR3180588} to all $\mathsf{HL}$ spaces with every point being $G_{\delta}$,\footnote{This assumption is stronger than $T_1$.} that:

\begin{fact}[\cite{MR980949} 8.12]
 Every Hausdorff $\mathsf{HL}$ space is of size less or equal to $\mathfrak{c}$.\hfill $\square$
\end{fact}

Fix a topological space $X$, $F\subseteq X^n$, $\mathcal{B}_0$ a basis for the topology of $X$, $\mathcal{B}_0^{n}$ the basis of the topology of $X^{n}$ given by products of elements of $\mathcal{B}_0$ and $K_0$ an open coloring (partition) over $X$.

\begin{definition}[Abraham-Rudin-Shelah \cite{abraham1985consistency}]
We say that $\langle M_{\xi} :\xi\in \omega_{1}\rangle $ is a \emph{tower of models for $X$, $F$, $\mathcal{B}_0$ and $K_0$} if and only if
\begin{enumerate}
	\item There is $\theta>\mathfrak{c}$ such that for all $\xi<\omega_1$, $M_{\xi}$ is an elementary submodel of $H(\theta)$ with $|M_{\xi}|=\aleph_0$, $E, F, K_0, \mathcal{B}_0^{n}\in M_{\xi}$ for all $n\in\omega$ and $M_{\xi}$ countably closed, \emph{i.e.}, if $p\in M_{\xi}$ is countable, then $p\subseteq M_{\xi}$.
	\item $M_{\xi}\in M_{\xi+1}$.
	\item $M_{\gamma}=\bigcup_{\xi\in \gamma}M_{\xi}$ if $\gamma$ is a limit ordinal.
\end{enumerate}
\end{definition}

Notice that if $\mathfrak{c}<\theta$, then $2^{\omega}\in H(\theta)$ and $2^{\omega}\subseteq H(\theta)$. Because of these, given a countable subset $\mathcal{B}$ of the basis, all the open and closed sets of $X^{n}$ generated by $\mathcal{B}$ are in $H(\theta)$ (each one codified by a real number). Furthermore, we know that for $\mathsf{HL}$ spaces all open sets are a countable union of open basic sets. This means that, if $X^{n}$ is $\mathsf{HL}$ and $\theta$ is such that $w(X)^{\aleph_0}\leqslant (2^{\mathfrak{c}})^{\aleph_0}=2^{\mathfrak{c}}<\theta$, then $H(\theta)$ actually contains all open and closed sets of $X^{n}$.

\medskip

Under $\mathsf{CH}$, since $\aleph_1^{\aleph_0}=\aleph_{1}=\mathfrak{c}$ we can ask for a tower with $X$ (a Hausdorff strong $\epsilon$-space with $w(X)\leqslant \aleph_{1}$), $K_0$, $\mathcal{B}$ (a countable collection of basic open sets which union is $K_0$) and $\mathcal{B}_{0}^{n}$ a basis for $X^{n}$ to have all the open and closed sets of $X^{n}$ at some level in the tower. Nevertheless, there are $\mathsf{HL}$ Hausdorff spaces with weight bigger than $\mathfrak{c}$. To solve that, we can do the following trick inspired in the submodel topology:

\begin{lemma}\label{Small weight s-epsilon}
Given a Hausdorff strong $\epsilon$-space $\langle X, \tau \rangle$, there exists a topology $\tau'\subseteq \tau $, at most of size continuum, and subspace $X'\subseteq X$, at most of size continuum, such that $\langle X', \tau'\rangle$ is a Hausdorff strong $\epsilon$-space.
\end{lemma}

\begin{proof}
Let $\theta$ be such that $2^{\mathfrak{c}}<\theta$. Then there is an homoemorphic copy of $\langle X, \tau \rangle$ such that $\langle X, \tau \rangle\in H(\theta)$ and $X,\tau\subseteq H(\theta)$.


Let $N$ be an elementary submodel of $H(\theta)$ of size $2^{\aleph_{0}}$, countably closed, closed under countable sequences (\emph{i.e.}, $N^{\omega}\subseteq N$, this is possible since $\mathfrak{c}^{\omega}=\mathfrak{c}$) and such that $\langle X, \tau \rangle\in N$.


Let $X'=X\cap N$ and $\tau'=(\tau\cap N)\restrict X'$. First we will show that $\tau'$ is a topology. Clearly, it has $X'$, $\emptyset$ and it is closed under finite intersection. To show that it is closed under any size unions of sets first notice that given a countable sequence $U_{n}$ of elements in $\tau'$ we have that $\bigcup_{n\in \omega} U_{n}\in\tau'$. Now, let $\bigcup_{i\in I} U_{i}$ where $U_{i}\in \tau'$. We know that $U_{i}=V_{i}\cap X'$ for some open set $V_{i}$. Let $V=\bigcup_{i\in I} V_{i}\subseteq X$. Since $X$ is $\mathsf{HL}$, there exist $i_n\in I$ such that $V=\bigcup_{n\in \omega} V_{i_{n}}$. Then we have that $V_{i_n}\in N$ and $\{V_{i_n}:n\in \omega\}\in N$. Therefore, 
\[\bigcup_{i\in I} U_{i}=\bigcup_{i\in I} V_{i}\cap X'=\bigcup_{n\in \omega} V_{i_{n}}\cap X'\in (\tau\cap N)\restrict X'= \tau'.\]

Notice that, the above proof also shows that any open subspace is Lindel\"of. So $\langle X', \tau'\rangle$ is $\mathsf{HL}$ and Hausdorff (by elementarity). The same proof, applied to $(X')^n$, shows that $X'$ is an strong $\epsilon$-space.
\end{proof}

Suppose that we have an space $X$ that is Hausdorff and strongly $\epsilon$; and that $K_0$ is the open part of a partition of $[X]^{2}$. If we make sure that $K_0'=K_0\cap [X']^{2}\in [\tau']^{2}$, then $K_0'$ is an open coloring of $\langle X', \tau'\rangle$. Furthermore, any uncountable homogeneous set of $X'$ would also be an uncountable homogeneous set of $X$. So, this give us the liberty to only work with Hausdorff strong $\epsilon$-spaces of size and weight less or equal to $\mathfrak{c}$.

\begin{definition}[Abraham-Rudin-Shelah \cite{abraham1985consistency, kunen2014set}]
Let $X$ be a topological space and $K_0$ the open part of a coloring over $X$. Then we define the folowing forcing:
\[\mathds{P}_{X, K_0}=\{p\in [X]^{<\omega}: \forall x,y\in p (x=y \vee \{x,y\}\in K_0\}.\]
\end{definition}

\begin{lemma}\label{SOCA ccc CH}
$(\mathsf{CH})$ Let $X$ be an uncountable Hausdorff strong $\epsilon$-space such that $w(X)=\mathfrak{c}$ and let $K_0$ be the open part of an open partition. If $X$ has no uncountable homogeneous subset for the partition $K_0\cup K_1$ then there is an uncountable set $X_0\subseteq X$ such that $\mathds{P}_{X_{0}, [X_{0}]^{2} \cap K_0}$ is ccc. 
\end{lemma}

\begin{proof}

Let $\mathcal{B}_0$ be a basis for the topology of $X$, $\mathcal{B}$ a collection of disjoint basic open sets such that $K_0$ can be represented as the union of products of elements of $\mathcal{B}$ and let $\langle M_{\xi} : \xi<\omega \rangle$ be a tower of models for $X$, $\emptyset$, $\mathcal{B}_0$ and $K_0$ such that $\mathcal{B}\in M_0$ (to ensure that $K_0$ is open). Since $\aleph_1^{\aleph_0}=\aleph_{1}=\mathfrak{c}$, we know that all elements of $X^{n}$ (an $\mathsf{HL}$ space) and all open and closed sets will appear at some level of the tower.

\medskip

Let $X_{0}=\{x_{\xi}: \xi<\omega_{1}\}$ such that for each $\xi<\omega_{1}$ we have that $x_{\xi}\in M_{\xi+1}\setminus M_{\xi}$. In other words, the first appearance of $x_{\xi}$ is at level $\xi+1$.

\medskip

Suppose that we have an uncountable set $A\subseteq \mathds{P}_{X_{0}, [X_{0}]^{2} \cap K_0}$. Using the $\Delta$-system Lemma, we can assume that all the elements of $A$ have the same size, say $n$, and share a root $r$. Since two elements of this forcing are compatible if and only if its union is $K_0$-homogeneous, we can assume that $r=\emptyset$.

\medskip

Notice that the elements of $X_{0}$ have a natural order (the level in which they appear in the tower). Using that, we can uniquely associate a member of $A$ with an element of $X^{n}$. We just order them according to their order of appearance in the tower (starting from the smallest one). Our next step will be to present some notation and shrink $A$ a little.

\medskip

For $x\in X_{0}$ we let $ht(x)=\xi$ if and only if $x=x_{\xi}$. For $p\in A\subseteq X^{n}$ we write it as $p=p_{\alpha}^{\frown}x^{n-1}_{\alpha} $ where $p_{\alpha}=(x^{0}_{\alpha},..., x^{n-2}_{\alpha})\in X^{n-1}$ and $ht(x^{n-1}_{\alpha})=\alpha$. We shrink $A$ so that $\alpha<ht(x^{0}_{\beta})$ for any $\beta>\alpha$. Since $A$ is uncountable, $\Gamma=\{\alpha: \exists (p\in A) p=p_{\alpha}^{\frown}x^{n-1}_{\alpha}\}$ is unbounded in $\omega_{1}$.

\medskip

Given that each element of $A$ is $K_{0}$-homogeneous, there exists basic open $U^{i}_{\alpha}$, pairwise disjoint, such that $x_{\alpha}^{i}\in U^{i}_{\alpha}$ and $U^{i}_{\alpha}\otimes U^{j}_{\alpha}\subseteq K_0$ whenever $i\neq j$. Since $X^{n}$ is $\mathsf{HL}$, and $\{\prod_{i=0}^{n-1}U^{j}_{\alpha}: \alpha\in \Gamma\}$ is an open cover of $A$, there are $\alpha_{k}$ for $k\in \omega$ such that $\{\prod_{i=0}^{n-1}U^{j}_{\alpha_k}: k\in \omega\}$ covers $A$.

\medskip

Since $A$ is uncountable, there must be $k_0\in \omega$ such that $\prod_{i=0}^{n-1}U^{j}_{\alpha_{k_0}}\cap A$ is uncountable. Without loss of generality, we can shrink $A$ such that for all $\alpha, \beta \in \Gamma$ and $j<n$, we have that $U^{j}_{\alpha}=U^{j}_{\beta}$. We will call these unique open sets $U^{j}$.  Notice that, with this reduction, we know that given $i\neq j$ and $\alpha, \beta\in \Gamma$, $\{x^{i}_{\alpha}, x^{j}_{\beta}\}\in K_{0}$. This means that we only need to find $\alpha_{0}, \beta_{1}\in \Gamma$ such that, for all $i\in n$, $\{x^{i}_{\alpha_0}, x^{i}_{\beta_1}\}\in K_0$.

\medskip

We do this with an induction over $n$:

\medskip

For $n=1$, we have that $\bigcup A=\{x^{0}_{\alpha}: \alpha\in \Gamma\}$ is an uncountable set of $X$. By hypothesis, this is not an homogeneous set, so there exists $\alpha_0, \beta_{1}$ such that $\{x^{0}_{\alpha_0}, x^{0}_{\beta_1}\}\in K_0$. In other words, we have that $\{x^{0}_{\alpha_0}\}\cup\{ x^{0}_{\beta_1}\}\in \mathds{P}_{X_{0}, [X_{0}]^{2} \cap K_0}$.

\medskip

Assuming that it is true for $n$, we will show it for $n+1$. Let $F$ be the closure of $A$ in $X^{n+1}$. Thanks to $\mathsf{CH}$ and $X^{n+1}$ being $\mathsf{HL}$, we know that there exists $\gamma\in \omega_1$ such that $F\in M_{\gamma}$. Then, for all $\alpha\in \Gamma$, $\alpha >\gamma$, we have 
\[F_{\alpha}=\{y\in X: p_{\alpha}^{\frown}y\in F \}\in M_{\alpha}.\]

Notice that $x^{n}_{\alpha}\notin M_{\alpha}$, since $ht(x^{n}_{\alpha})=\alpha$, so $F_{\alpha}\not\subseteq M_{\alpha}$. This shows that $F_{\alpha}\subseteq X$ is uncountable. Using our hypothesis, we know that there are $y^{0}_{\alpha}, y^{1}_{\alpha}\in F_{\alpha}$ such that $\{y^{0}_{\alpha}, y^{1}_{\alpha}\}\in K_0$. Since $K_0$ is open, there are disjoint open sets $V^{0}_{\alpha}, V^{1}_{\alpha}\in \mathcal{B}_0$ such that $\{y^{0}_{\alpha}, y^{1}_{\alpha}\}\in V_{0}^{\alpha}\otimes V_{1}^{\alpha}\subseteq K_0$.  This means that $\bigcup_{\alpha \in \Gamma\setminus (\gamma+1)}\{(y^0_\alpha, y^1_\alpha)\}\subseteq \bigcup_{\alpha \in \Gamma\setminus (\gamma+1)}V_{\alpha}^{0}\times V_{\alpha}^{1}$. Since $X^2$ is $\mathsf{HL}$, we can find $\Lambda \subseteq \Gamma\setminus (\gamma+1)$ uncountable such that $V^{j}_{\alpha}=V^{j}_{\beta}$ for all $j\in \{0,1\}$ and $\alpha, \beta \in \Lambda$. We will call $V^{n}_{j}$ the unique basic open such that $V^{j}_{\alpha}=V^{j}_{\beta}$ for all $j\in \{0,1\}$ and $\alpha, \beta \in \Lambda$.

\medskip

Focus on $B=\{p_{\alpha}: \alpha\in \Lambda\}\subseteq X^{n}$. Since $\Lambda$ is uncountable, $B\subseteq X^{n}$ is also uncountable. Using our induction hypothesis, there exists $\delta_0, \delta_1\in \Lambda$ such that $p_{\delta_0}\cup p_{\delta_1}\in  \mathds{P}_{X_{0}, [X_{0}]^{2} \cap K_0}$. In other words, $\{x^{i}_{\delta_{0}}, x^{i}_{\delta_{1}}\}\in K_0$ for all $i\in n$. Since $K_0$ is open, there are $V^{i}_{j}$ disjoint basic open sets, for $i\in n$ and $j\in \{0,1\}$, such that $x^{i}_{\delta_{j}}\in V^{i}_{j}$ and $V^{i}_{0}\otimes V^{i}_{1}\subseteq K_0$.

\medskip

Finally, notice that $p_{\delta_{j}}^{\frown}y^{j}_{\delta_{j}}\in \prod_{i=0}^{n}V^{i}_{j}$. This shows that $\prod_{i=0}^{n}V^{i}_{j}\cap F\neq \emptyset$. Being $F$ the closure of $A$ in $X^{n+1}$ and $\prod_{i=0}^{n}V^{i}_{j}$ an open set, there must be $\alpha_{0}, \beta_1\in \Gamma$ such that $p_{\alpha_0}^{\frown}x^{n}_{\alpha_0}\in \prod_{i=0}^{n}V^{i}_{0}$ and $p_{\beta_1}^{\frown}x^{n}_{\beta_1}\in \prod_{i=0}^{n}V^{i}_{1}$. This implies that, for all $i\in n$, $\{x^{i}_{\alpha_0}, x^{i}_{\beta_1}\}\in K_0$.

\medskip

Since we have found two compatible members in $A$, it is not an antichain. Therefore, $\mathds{P}_{X_{0}, [X_{0}]^{2} \cap K_0}$ is ccc.
\end{proof}

\begin{theorem}
If $\mathsf{ZFC}+\mathsf{CH}$ is consistent, then $\mathsf{ZFC}$+$\mathsf{SOCA}(s\text{-}\epsilon)$ is also consistent.
\end{theorem}

\begin{proof}
To show the consistency of $\mathsf{SOCA}(s\text{-}\epsilon)$ is it enough to follow the presentation given in either \cite{MR801036} or \cite{MR2905394}.
\end{proof}


It turns out that Lemma \ref{SOCA ccc CH} does all the heavy lifting for two more consistency results.

\medskip

First, we will work with  $\mathsf{OGA}$ for strong $\epsilon$-spaces. Let $X$ be a Hausdorff strongly $\epsilon$-space and $K_0$ the open part of a partition of $[X]^{2}$. Notice that if $H\subseteq X$ is a $K_1$-homogeneous set then the closure of $H$ is also $K_1$-homogeneous. Therefore, being cover by countably many closed $K_1$-homogeneous sets, is a property that is inherit in a submodel closed under countable sequences. Adding this observation to the proofs of Lemma \ref{Small weight s-epsilon} we have that:

\begin{corollary}\label{Small weight s-epsilon not covered}
Given a Hausdorff strong $\epsilon$-space $\langle X, \tau \rangle$  and $K_0$ the open part of a partition of $[X]^{2}$ such that $X$ is not cover by countably many closed $K_1$-homogeneous sets. Then, there exists a topology $\tau'\subseteq \tau $, at most of size continuum, and a subspace $X'\subseteq X$, at most of size continuum, such that $\langle X', \tau'\rangle$ is a Hausdorff strong $\epsilon$-space, $K_0\cap[X']^{2}$ is open and $X'$ is not cover by countably many closed $K_1\cap [X']^{2}$-homogeneous sets.\hfill $\square$
\end{corollary}

\medskip

\begin{lemma}\label{Diagonalizing countable covers}
$(\mathsf{CH})$ Given a Hausdorff strong $\epsilon$-space $\langle X, \tau \rangle$  and $K_0$ the open part of a partition of $[X]^{2}$ such that $X$ is not cover by countably many closed $K_1$-homogeneous sets. Then, there exists a size continuum topology $\tau'\subseteq \tau $ and a size continuum subspace $X'\subseteq X$ such that $\langle X', \tau'\rangle$ is a Hausdorff strong $\epsilon$-space, $K_0\cap[X']^{2}$ is open and $X'$ has no $K_1$-uncountable homogeneous sets.
\end{lemma}

\begin{proof}
Let $\langle X^{\ast}, \tau^{\ast}\rangle$ be the subset and topology given by Lemma \ref{Small weight s-epsilon}.

\medskip

Since $X^{\ast}$ is $\mathsf{HL}$ with $w(X^{\ast})=\mathfrak{c}=\aleph_1$, then it has $\aleph_1^{\aleph_0}=\mathfrak{c}=\aleph_1$ many closed sets. Let $\{F_{\xi}: \xi<\omega_{1}\}$ be an enumeration of all closed $K_{1}$-homogeneous sets. Since for any $\alpha<\omega_{1}$ we have that $X^{\ast}\neq \bigcup_{\xi<\alpha}F_{\xi}$, we can create a subset $X'\subseteq X$ of size $\aleph_1$ such that $|X'\cap F_{\xi}|\leqslant |\xi|\leq\aleph_0$, Of course, we would use recursion to define such a set.

\medskip

Notice that $X'$ and $\tau'=\tau^{\ast}\restrict X'$ satisfy the desired conditions.
\end{proof}

So, now, we can use Lemma \ref{SOCA ccc CH} over $\langle X', \tau'\rangle$ to create an adequate ccc forcing. Finally, we have that:

\begin{theorem}
If $\mathsf{ZFC}$+$\diamondsuit_{\omega_{2}}$ is consistent, then $\mathsf{ZFC}$+$\mathsf{OGA}(s\text{-}\epsilon)$ is also consistent.
\end{theorem}

\begin{proof}
Actually, \cite{MR2905394} has an exposition on how to use Lemma \ref{SOCA ccc CH} with $\diamondsuit_{\omega_2}$ to obtain this result. For completeness, we add an sketch of the proof:

\medskip

\begin{proofsketch}
Remember that all Hausdorff $\mathsf{HL}$ spaces are of size less or equal to $\mathfrak{c}$ and that $\mathfrak{c}=\aleph_2$ in the final model. Since being an strongly $\epsilon$-space is clearly hereditarily, we use $\diamondsuit_{\omega_{2}}$ to guess a sequence of  strongly-$\epsilon$-spaces of size $\aleph_1$ with a coloring over it with no $\sigma$-$K_1$-homogenous cover. This $\diamondsuit$-sequence of length $\aleph_2$ will ensure that any strongly-$\epsilon$ space in the final model either is cover by countably many $K_1$-homogeneous sets or it contains one of the topological spaces in the sequence.
\end{proofsketch}
\end{proof}

To show the consistency of $\mathsf{SOCA}$ for strong $\epsilon$-spaces with a big continuum, we will use axiom $\mathsf{A1}$ as in \cite{abraham1985consistency}. That axiom allows to combine $\mathfrak{c}$ many towers of models into a single one and it survives ccc (essentially, it adds a fast club). So, following the same proofs as in Lemma \ref{SOCA ccc CH} and the forcing of Abraham-Rudin-Shelah \cite{abraham1985consistency}, we have:

\begin{lemma}\label{SOCA ccc A1}
$(\mathsf{A1})$ Let $X$ be an uncountable Hausdorff strong $\epsilon$-space such that $w(X)=\mathfrak{c}$ and let $K_0$ be the open part of an open partition. If $X$ has no uncountable homogeneous subset for the partition $K_0\cup K_1$ then there is an uncountable set $X_0\subseteq X$ such that $\mathds{P}_{X_{0}, [X_{0}]^{2} \cap K_0}$ is ccc.\hfill $\square$
\end{lemma}

\begin{theorem}
If $\mathsf{ZFC}$+$\mathsf{A1}$+$2^{\aleph_1}=\kappa$ is consistent, then $\mathsf{ZFC}$+$\mathsf{SOCA}(s\text{-}\epsilon)$+$\mathfrak{c}=\kappa$ is also consistent.\hfill $\square$
\end{theorem}

Analyzing all of our consistency results, we observed that the key lemma for the club method strategy was Lemma \ref{SOCA ccc CH}.\footnote{Remember that the Songfrey line does satisfy $\mathsf{OGA}$ although its square is not Lindel\"of and we cannot use Lemma \ref{SOCA ccc CH} on it.} Looking for ways of improving such Lemma, we noticed that the property of finitely powerfully  $\mathsf{HL}$ (strongly $\epsilon$-space) was only use in three places:

\begin{enumerate}
    \item To reduce $A$ in order to have all $U^{i}_{\alpha}$ being the same. 
    \item To ensure that $K_0$ is the union of countably many basic open sets.
    \item To ensure that there are only continuum many open and closed sets of $X^{n}$
\end{enumerate}

\medskip

It seems that number 2 is unavoidable. Now, for number 1, even if $A$ is in $X^{n}$, we are only looking at relations in $X^{2}$. So, there is a chance that, being careful, $X^{2}$ being $\mathsf{HL}$ is the only requirement to reduce $A$ in order to have all $U^{i}_{\alpha}$ being the same.

\medskip

Finally, number 3 is the trickiest. Although $\aleph_{1}^{\aleph_0}=\mathfrak{c}$ and Lemma \ref{Small weight s-epsilon} helps a lot ensuring a $\aleph_1$ size base, if $X^{n}$ is not $\mathsf{HL}$, then it has $\aleph_1^{\aleph_1}$ many open and closed sets. This could work out using axiom $\mathsf{A1}$ and having $2^{\aleph_0}=2^{\aleph_1}=\kappa>\aleph_1$, but it is not clear how to deal with that in a context with $\mathsf{CH}$.

\medskip

We explored the possibility of only looking at a ``countable closure'', \emph{i.e.}, the closure of all the sets which are complements of open sets created by countably many basic opens, but it is not clear how this will work out.

\section{Under the glasses of \texorpdfstring{$\mathsf{PFA}$}{Lg}}

In Section \ref{OGA and SOCA epsilon strong Consistency} we gave $X$ a really strong covering hypothesis, but being Hausdorff was enough. One way to weaken the covering hypothesis is to make the space $T_{3}$ (regular). In that case, we have the following known result.

\begin{fact}[\cite{MR493957}]\label{HL or HS}
Assume that $X$ is a regular topological space of countable spread. Then:
\begin{enumerate}
    \item If $X$ is not $\mathsf{HL}$ then it has an $\mathsf{S}$-space as a subspace.
    \item If $X$ is not $\mathsf{HS}$ then it has an $\mathsf{L}$-space as a subspace.\hfill $\square$
\end{enumerate}
\end{fact}

\begin{fact}[Todor\v{c}evi\'{c} \cite{MR980949}]
$\mathsf{PFA}$ implies that there are no $\mathsf{S}$-spaces. In particular, all regular topological spaces with countable spread are $\mathsf{HL}$.\hfill $\square$
\end{fact}

\begin{definition}
We say that a space has the \emph{finitely powerfully countable spread property} if and only all its finite powers have countable spread. 
\end{definition}

\begin{lemma}[$\mathsf{PFA}$] 
Given $X$ a regular space with countable spread such that $X^{2}$ has countable spread, then it has finitely powerfully countable spread.
\end{lemma}

\begin{proof}
We will only show that $X^3$ has countable spread for notational simplicity. Nevertheless, the way to extend this proof will be clear.

\medskip

The definition of $A\subseteq X^3$ being discrete is that for each $a\in A$ there exist an open set $U_a\subseteq X^3$, $a\in U_a$, such that $|U_a\cap A|=1$ ($U_a\cap A=\{a\}$). Then, if we show that given any open cover of $A$ there is a basic open $U$, a member of the cover, such that $|U\cap A|>1$, we show that $A$ is not discrete. Also, it is enough to work with cover which member are only the basic open sets.

\medskip

Assume that we have $A=\{x^{\alpha}=(x^{\alpha}_1, x^{\alpha}_2, x^{\alpha}_3): \alpha <\omega_1\}\subseteq X^3$. For each $\alpha<\omega_1$ let $U^{\alpha},V^{\alpha}, W^{\alpha}\subseteq X$ be basic open sets such that $x^{\alpha}_1\in U^{\alpha},x^{\alpha}_2\in V^{\alpha}, x^{\alpha}_3\in W^{\alpha}$.

\medskip

Look at the set $A'=\{(x^{\alpha}_1, x^{\alpha}_2): x^{\alpha}\in A\}\subseteq X^2$. Since $X^2$ has countable spread, using $\mathsf{PFA}$, it is $\mathsf{HL}$. This implies that we can find open sets $U,V\subseteq X$ such that $U=U^{\alpha}$ and $V=V^{\alpha}$ for uncountably many $\alpha$. In other words, $U\times V \cap A'$ is uncountable. Let $\Gamma=\{\alpha: (x^{\alpha}_1, x^{\alpha}_2)\in U\times V\}$ and let $B=\{x^{\alpha}: \alpha\in \Gamma\}$.

\medskip

Notice that the set $\{U\times V\times W^\alpha: \alpha \in \Gamma\}$ is an open cover of $B$, an uncountable set. Now, look at $B^{\ast}=\{(x^{\alpha}_2, x^{\alpha}_3): x^{\alpha}\in B\}\subseteq X^2$. Again, we can find $W$ such that for uncountable many $\beta\in \Gamma$ we have $W^{\beta}=W$. Notice that $V\times W\cap B^{\ast}$ is uncountable.

\medskip

Finally, we have that $U\times V\times W\cap B\subseteq U\times V\times W\cap A$ is uncountable. So, $A$ is not a discrete subspace of $X^3$. Therefore, $X^3$ has countable spread.
\end{proof}

\begin{theorem}
$\mathsf{PFA}$ implies $\mathsf{OGA}($Finitely powerfully countable spread regular spa\-ces$)$.
\end{theorem}

\begin{proof}
By all the facts above, $\mathsf{PFA}$ implies that finitely powerfully countable spread regular spaces are strongly $\epsilon$-spaces. After that, using the ``collapsing the continuum trick'' and Lemma \ref{SOCA ccc CH}, we are done.
\end{proof}

\begin{corollary}
$\mathsf{PFA}$ implies $\mathsf{OGA}(X)$ for all regular spaces $X$ such that $X$ and $X^2$ have countable spread.\hfill $\square$
\end{corollary}

It turns out that $\mathsf{PFA}$ can give us more information about how this class of spaces looks like (when the axiom is present):

\begin{fact}
$MA_{\omega_1}$ $($a consequence of $\mathsf{PFA})$ implies that an space $X$ is strong $\epsilon$ if and only if $X^n$ is $\mathsf{HS}$ for all $n$.\hfill $\square$
\end{fact}

\begin{fact}\label{No first L spaces}
$MA_{\omega_1}$ implies that there are no first countable $\mathsf{L}$-spaces.\hfill $\square$
\end{fact}

\begin{remark}
Assuming $\mathsf{PFA}$, all regular spaces with countable spread are either strongly $\epsilon$-spaces or their square has an uncountable discrete subspace. Notice that each uncountable subspace of $\mathbb{S}$ enters in the second category and Theorem \ref{ocasor} presents an strategy to show that they satisfy $\mathsf{OGA}$.
\end{remark}

\begin{conjecture}\label{final conjecture}
Under $\mathsf{PFA}$, all regular first countable spaces with countable spread are either strongly $\epsilon$-spaces or has a countable c.d.$\pi$-base.
\end{conjecture}

One might be tempted to remove the requirement of first-countability in this conjecture, but this is not possible by Moore's $\mathsf{ZFC}$ $\mathsf{L}$-space \cite{MR2220104}.

\section{Questions and conclusions}

Conjecture \ref{final conjecture} can also be re-interpreted in a combinatorial way. As remarked in the definition of $\epsilon$-spaces, this property is equivalent to say that every open $\omega$-cover (\emph{i.e.}, a cover such that for all finite subsets of the space, there is an element of the cover that contains the set) has a countable subcover. On the other hand, countable c.d.$\pi$-base means that, given a base of the topology $\mathcal{B}$, you can substitute it for a countable set $\mathcal{V}$ such that if $U\in \mathcal{B}$, then $U\setminus \mathsf{Int}_{\mathcal{V}}(U)$ is countable (and the definition could be strengthen to ask for finiteness).

\medskip

In this way, $\epsilon$-spaces and spaces with a countable c.d.$\pi$-base present a tension where the first ones can reduce \emph{outer} covers to be countable and the second one change the \emph{interior} of basic open sets up to a countable difference. Given the nice properties that $\mathsf{PFA}$ provoke in regular topologies, we wonder whether or not $\mathsf{PFA}$ is powerful to imply that every regular space has a nice outer behavior or a nice inner one. 

\medskip

Further investigation of countable c.d.$\pi$-bases is necessary and some open questions remain open:

\begin{question}
Which topological properties imply the existence of a countable c.d.$\pi$-base?
\end{question}

\begin{question}
Is it the case that if $X$ has a countable c.d.$\pi$-base and $X$ is not second countable then $X^{2}$ is not Lindel\"of (or does not have countable spread)?
\end{question}


In a different topic, our main focus of study was extending the class of spaces where the Coloring Axioms are either $\mathsf{ZFC}$ true or where it is consistent for the axioms to hold. As explain in the introduction, the ultimate goal of this expansion hasn't been achieved. Nevertheless, we can present simpler questions that, most likely, will help open the path:

\begin{question}
Let $L$ be a Hausdorff $\mathsf{L}$-space. Does $\mathsf{OGA}(L)$ holds? Does $\mathsf{PFA}$ makes a difference regarding the answer?
\end{question}

\begin{question}
 Is it relatively consistent with $\mathsf{ZFC}$ that $\mathsf{OGA}_{X}$ holds for $X \in \mathcal{S}^{*}$? 
\end{question}

Moreover, in virtue of Corollary \ref{equivalencesorgenfrey}:

\begin{question}
 Is $\mathsf{OGA}$ equivalent to $\mathsf{OGA}_{X}$ for $X \in \mathcal{S}^{*}$?
\end{question}

Let $\mathcal{S}_{B}^{*}$ be the subclass of $\mathcal{S}^{*}$ consisting of all spaces of the form $\langle\omega^{\omega},\tau_{X}\rangle$. By Theorem \ref{Th:Continuous image of Baire} (2), each element of $\mathcal{S}^{*}$ is a continuous image of some element of $\mathcal{S}_{B}^{*}$. Clearly  $\mathbb{S}$, $\omega^{\omega}$ and $\langle\omega^{\omega}, \tau_{GH}\rangle$ belong to $\mathcal{S}_{B}^{*}$:

\begin{question}
 Is there a surjective finite basis for $\mathcal{S}_{B}^{*}$?
\end{question}

\begin{question}
 Is each $Z \in \mathcal{S}$ a continuous image of a subspace $Y$ of $X \in \mathcal{S}_{B}^{*}$? 
\end{question}

\bibliographystyle{amsplain}

\bibliography{Biblioart.bib}
\end{document}